\RequirePackage{fix-cm}
\documentclass[smallextended]{svjour3}       
\smartqed  
\usepackage{graphicx}
 \usepackage{mathptmx}      
\usepackage{amssymb,amsmath}
\usepackage{amsmath,amscd}
\usepackage{amsthm}
\usepackage{overpic}
\usepackage{color}
\usepackage{standalone}
\usepackage{stmaryrd}

\usepackage{mathtools}
\usepackage{amsfonts}
\usepackage{mathrsfs}
\usepackage{tikz}
\usepackage[utf8]{inputenc}
\usepackage{graphicx}
\usepackage{comment}

\begin{document}

\title{Toroidal Cartesian Products Where One Factor is $3$-Connected}


\author{Elizabeth Badgett       \and Christian Millichap* \and Kenta Noguchi}


\institute{Elizabeth Badgett \at
              Department of Mathematics, Furman University, Greenville, SC 29613 \\
              \email{beppybadgett@gmail.com}           
           \and
           Christian Millichap \at
          Department of Mathematics, Furman University, Greenville, SC 29613 \\
	  \email{christian.millichap@furman.edu} \\
	  ORCID: 0000-0001-8265-1437
	  \and
	  Kenta Noguchi \at
	  Department of Information Sciences, Tokyo University of Science, Noda, Japan \\
	  \email{noguchi@rs.tus.ac.jp} \\
	  ORCID: 0000-0002-7790-0782
}

\date{Received: date / Accepted: date}

\maketitle

\begin{abstract}
In this paper,  we show that if $G$ is $3$-connected, then  the Cartesian product of graphs $G \boxempty H$ embeds on the torus  if and only if $G$ is outer-cylindrical and $H$ is a path on two vertices, $P_2$. As a by-product of our work, we also show that $K_{4} \boxempty P_{3}$ has genus two. 
\end{abstract}

\section{Introduction}
\label{sec:intro}

Classifying embeddings of graphs on surfaces is a well-studied topic in topological graph theory. Here, we are interested in embeddings of a \textbf{Cartesian product of two graphs $G$ and $H$}, denoted $G \boxempty H$, which is defined to have vertex set $V(G) \times V(H)$ and a pair of vertices $(u,v), (u',v') \in V(G) \times V(H)$ determine an edge in $G \boxempty H$ if and only if $u = u'$ and $vv' \in E(H)$ or $uu' \in E(G)$ and $v=v'$. We refer to $G$ and $H$ as the \textbf{factors} of $G \boxempty H$. Every $G \times \{h\}$ with $h \in V(H)$ is called a \textbf{$G$-fiber} of $G \boxempty H$, and $H$-fibers are defined analogously. We refer the reader to \cite{ImKlRa2008} for further background on graph Cartesian products. The fiber structures of graph Cartesian products can sometimes be leveraged to efficiently build embeddings and determine  genera of such graphs; see \cite{MiSa2022}, \cite{MoPiWh1990}, \cite{Pi1980}, \cite{Pi1982}, \cite{Pi1989},  \cite{Sun2023}, \cite{Wh1970}, \cite{Wh1971} for some examples.  In particular, it would be interesting to know if one can determine the set of all Cartesian products of graphs that embed on a fixed surface.  Such a classification was established for \textbf{planar} graphs, that is, graphs that embed in the plane, in the work of Behzad and Mahmoodian \cite{BeMa1969}.  More recently, the second author along with Abell and McDermott have provided a complete classification of Cartesian products of graphs that  embed in the projective plane in \cite{AbMcMi2025}.



In this paper, we are interested in classifying which Cartesian product of  graphs are \textbf{toroidal}, i.e., embed on the torus. While the classifications of planar and projective planar Cartesian products were reasonably straightforward, we expect the classification of toroidal Cartesian products to be far more complicated. To start,  the work of Kuratowski \cite{Ku1930} and Wagner \cite{Wa1937} showed there are only two minor-minimal obstructions to planarity, namely $K_5$ and $K_{3,3}$, and Archdeacon in \cite{Ar1980} showed that there is a complete list  of 35 graphs that are minor-minimal obstructions to embedding on the projective plane. However, the list of minor-minimal obstructions for embedding on the torus is not completely known, and this list contains at least 17 thousand forbidden minors; see \cite{MyWo2018}. In addition, when restricting to Cartesian products of graphs, we see far more flexibility in toroidal embeddings than planar or projective planar embeddings with regards to (vertex) connectivity. If $G \boxempty H$ is planar or projective planar, then both $G$ and $H$ are each at most $2$-connected. This follows from the fact that $K_4 \boxempty P_2$ is both non-planar and non-projective planar, and any $3$-connected graph contains a $K_4$ minor; see Lemma \ref{lem:Tutte3connected}.  However, Example \ref{Example1} in Section \ref{sec:constructing3conP2} shows that there exist $G \boxempty P_2$ that are toroidal where $G$ is $4$-connected. Thus, even a partial classification of toroidal Cartesian products could be a far more complicated task than the planar and projective planar cases. 

Here, we provide some major first steps towards this classification problem, with the additional assumption that one factor is $3$-connected. In what follows, a graph $G$ is \textbf{outer-cylindrical} if there exists a planar embedding of $G$ such that the vertices of $G$ all lie on the boundary of at most two faces in this embedding. 

\begin{theorem}
    \label{thm:classification}
    Let $H$ and $G$ be simple, connected, nontrivial graphs, and suppose $G$ is $3$-connected. Then $G \boxempty H$ is toroidal if and only if $G$ is an outer-cylindrical graph and $H = P_2$.
\end{theorem}

The proof of Theorem \ref{thm:classification} is broken up into a few pieces.  In Proposition \ref{thm:Hrestricted}, we show that if $G$ is $3$-connected, then any toroidal $G \boxempty H$ must have $H = P_2$. This proof follows from the fact that $K_4 \boxempty P_3$ is not toroidal, which is proved in Theorem \ref{thm:K4P3genus}.  Proposition \ref{prop:OCboxP2} provides an explicit construction for building a toroidal embedding of $G \boxempty P_2$, where $G$ is outer-cylindrical. Then Proposition \ref{prop:toroidalimpliesOC} shows that if $G \boxempty P_2$ is toroidal, then $G$ must be an outer-cylindrical graph. This proof only requires some basic combinatorial topology of graph embeddings.

A significant by-product of our work occurs in Section \ref{sec:K4P3} where we prove that the genus of $K_4 \boxempty P_3$ is two. This work requires a careful combinatorial analysis of all potential embeddings of this graph on a torus, which employs techniques similar to \cite{BrSq1988}. In particular, even though the girth of $K_4 \boxempty P_3$ is three, one cannot hope for a triangular embedding since every edge in a $P_3$ fiber does not lie on a $3$-cycle. Thus, determining the genus of such a graph is not a straightforward task.


\section{Background}
\label{sec:Background}

In this paper, we assume all graphs are simple, connected, and non-trivial (in the sense of being nonempty and not a single vertex), unless otherwise noted. We use $S_{g}$ to denote an orientable, closed, connected surface of genus $g \in \mathbb{N} \cup \{0\}$. Given an \textbf{embedding} of a graph $G$ on a surface $S_{g}$ via a continuous injective function $f: |G| \rightarrow S_{g}$ where $|G|$ represents a geometric realization of $G$, we sometimes use  $f(G)$ to denote this embedding or just $G$ when the embedding is clear from context. If each component of $S_{g} \setminus f(G)$ is homeomorphic to an open disk, then we say this embedding is a \textbf{$2$-cell embedding}.  The \textbf{genus} of a graph $G$, denoted $\gamma(G)$, is the minimum $g \in \mathbb{N} \cup \{0\}$ such that $G$ embeds on $S_{g}$. An embedding of $G$ that realizes the minimal genus is called a \textbf{minimal embedding}.  We will mainly be interested in $2$-cell embeddings of graphs, and since the work of Youngs \cite{Yo1963} shows that each minimal embedding of $G$ is a $2$-cell embedding, we usually do not need to make a distinction when determining the genera of a graph. We refer the reader to \cite{GrTu2001} for further background on graph embeddings on surfaces. 

A $2$-cell embedding $f: |G| \rightarrow S_{g}$ provides a $2$-cell decomposition of $S_{g}$ with $|V(G)|$ vertices, $|E(G)|$ edges, and $|F|$ faces, where each face is a component of $S_{g} \setminus f(G)$. Given such an embedding, the Euler characteristic of $S_{g}$ is given by $2-2g = \chi(S_{g}) = |V(G)| - |E(G)| + |F|$. The following proposition is a well known lower bound on genus, coming from this Euler characteristic formula and analyzing the minimal number of edges on any face in an embedding of such a graph.  Recall that the \textbf{girth} of a  graph $G$ (that is not a tree) is the length of its shortest cycle.

\begin{proposition}
	\label{prop:Eulerbound}
	Let $G$ be a graph (that is not a tree) with girth $\alpha$. Then 
	$$\gamma(G) \geq 1 + \frac{|E(G)|(\alpha-2)}{2\alpha} - \frac{|V(G)|}{2}.$$
	
	Equality holds if and only if there exists an embedding of $G$ where every face is bound by an $\alpha$-cycle and every edge borders two distinct faces. 
\end{proposition}



Suppose we are given a $2$-cell embedding of $G$. We let $f_{i}$ denote the number of $i$-gon faces for this embedding. We use the terms triangle, quadrilateral, pentagon, and hexagon to refer to faces of a given embedding that are bound by a $3$-cycle, $4$-cycle, $5$-cycle, and $6$-cycle, respectively. We say that an embedding of a graph is a \textbf{triangular embedding} if every face is a triangle. We note the following basic but important facts about $2$-cell embeddings of $G$: \\

\begin{flushleft}
	\textbf{Fact 1:} Each vertex of $G$ has an open $2$-cell neighborhood in $S_g$.
	
	\textbf{Fact 2:} Each edge of $G$ is either incident to  two distinct faces or twice incident to one face.  
	
	\textbf{Fact 3:} If $G = H \boxempty P_{n}$ for some graph $H$, then each edge from a $P_n$ fiber borders two $k$-gon faces, where $k \geq 4$. 
	
	\textbf{Fact 4:} $2|E(G)| = \displaystyle\sum_{i=3}^{\infty} i f_i$. \\
\end{flushleft}

Facts 1 and 2 are general facts about $2$-cell embeddings. Fact 3 follows from Fact 2 and since any edge in a $P_n$ fiber is not contained in a $3$-cycle. Fact 4 follows from Fact 2. This collection of facts will play an important role in ruling out potential toroidal embeddings of graphs in Section \ref{sec:K4P3}.

A graph $G'$ is a \textbf{minor} of $G$ if $G'$ can be obtained by a series of vertex deletions, edge deletions, or edge contractions on $G$. Another way to find a lower bound for the genus of a graph $G$  is by finding the genus (or lower bound on the genus) for a subgraph or minor of $G$. Since we will be interested in Cartesian products of graphs in this paper, we first confirm that the product of minors from each factor results in a minor of the product.

\begin{proposition}
	\label{prop:minorproducts}
	If $G'$ and $H'$ are minors of $G$ and $H$ then $G'\boxempty H'$ is a minor of $G\boxempty H$.
\end{proposition}
\begin{proof}
	Let $G$ and $H$ be graphs and suppose $G'$ be a minor of $G$. Recall that the vertex sets and edge sets of $G\boxempty H$ are defined as $V(G\boxempty H)=V(G)\times V(H)$ and $E(G\boxempty H)= \{ \{(u,v), (u,v')\} : u \in V(G), vv' \in E(H)\} \cup \{ \{(u,v), (u',v)\} : uu' \in E(G), v \in V(H)\}$. It suffices to show $G'\boxempty H$ is a minor of $G\boxempty H$ under an edge deletion,  vertex deletion, or  edge contraction, which we now handle as three separate cases. 
	\\
	\underline{Case 1:} Let $G'=G - e$ for some $e = uu'\in E(G)$. Then $G' \boxempty H$ can be obtained from $G \boxempty H$ by deleting edge $\{(u,v), (u',v)\} \in E(G \boxempty H)$, for each $v \in V(H)$. 
	\\
	\underline{Case 2:} Let $G'=G - v_{0}$ for some $v_{0} \in V(G)$.  Recall that deleting some vertex $v_{0}$ also deletes all edges incident to $v_{0}$. Then we can obtain $G' \boxempty H$ from $G \boxempty H$ by deleting vertex $(v_{0}, v) \in V(G) \times V(H)$, for each $v \in V(H)$ (and deleting all its incident edges). 
	\\
	\underline{Case 3:} Given edge $e = uu'  \in E(G)$, we  contract the edge $e$ in $G$ to create $G'$ which now has a new vertex $a \in V(G')$ to which $u$ and $u'$ have been identified.  Then $G'\boxempty H$ can be obtained from $G\boxempty H$ by contracting  edge $\{(u, v), (u',v)\} \in E(G \boxempty H)$, for each $v \in V(H)$.  
\end{proof}

\begin{corollary}
	\label{cor:minorlowerbound}
	Suppose $G'$ and $H'$ are minors (or subgraphs) of graphs $G$ and $H$, respectively. Then $\gamma(G \boxempty H) \geq \gamma(G' \boxempty H')$.  
\end{corollary}


\section{The genus of $K_4 \boxempty P_3$}
\label{sec:K4P3}

In this section, we will show that $\gamma(K_4 \boxempty P_3) = 2$. Since every $3$-connected graph contains a $K_{4}$ minor (see Lemma \ref{lem:Tutte3connected}), this will play a pivotal role in classifying toroidal Cartesian products where one of the factors is $3$-connected. 

For the remainder of this section, set $G = K_4 \boxempty P_3$. We have that $|V(G)| = 12$ and $|E(G)| = 26$. In addition, there are twelve $3$-cycles in $G$, with four $3$-cycles in each $K_{4}$ fiber. In particular, no edge contained in a $P_3$ fiber is contained in a $3$-cycle.

\begin{lemma}
	\label{lem:trianglelimits} For any $2$-cell embedding of $G$, at most two $3$-cycles in each $K_4$ fiber are triangles, and so,  $f_{3} \leq 6$. \end{lemma}

\begin{proof}
	First, since $G$ has a girth of three, any triangle in a $2$-cell embedding of $G$ must be bound by a $3$-cycle. Suppose we are given a $2$-cell embedding of $G$ where there exist three triangles whose boundary $3$-cycles are all contained in the same $K_4$ fiber. Then there exists some vertex $v$ in this $K_{4}$ fiber where all three edges ($e_1$, $e_2$, and $e_3$) incident to $v$ in this $K_{4}$ fiber already border two faces in this embedding. Consider the unique edge $e \in E(K_4 \boxempty P_2)$ that is incident to $v$ and contained in a $P_3$ fiber. By Fact (2), $e$ must border exactly two faces in this embedding. However, any facial walk for such an embedding of $G$ that contains $e$ also contains one of the $e_i$ for $i=1,2,3$; otherwise, a vertex incident to $e$ would be degree $1$, which is not the case. However, this implies that some $e_i$ borders three faces, and so, violates Fact (1). Thus, any $2$-cell embedding of $G$ admits at most two triangles coming from any $K_4$ fiber. Since $3$-cycles of $G$ only occur within each $K_4$ fiber, we have that $f_{3} \leq 6$, as needed. 
\end{proof}

This lemma already places strong restrictions on potential embeddings of $G$. Suppose $G$ has a $2$-cell embedding on $S_{g}$, for some $g \in \mathbb{N} \cup \{0\}$. Then $2-2g = \chi(S_g) = 12 - 26 + |F|$, and so, $g=  8 - \frac{|F|}{2}$. From Fact 4, we know that given any $2$-cell embedding of $G$, we must have $2|E(G)| =  52 =\displaystyle\sum_{i=3}^{\infty} i f_i$. In order to maximize the number of faces in order to minimize genus, one would attempt to use six triangles in an embedding, and so, $34 = \displaystyle\sum_{i=4}^{\infty} i f_i$. This implies that any $2$-cell embedding of $G$ has at most $14$  faces. Thus $\gamma(G) \geq 1$.  Our main goal is to now prove that $\gamma(G) \neq 1$ by showing that an embedding of $G$ with $|F|=14$ is not possible.

\begin{proposition}
	\label{prop:toroidalembeddingcases}
	Suppose $G$ admits a $2$-cell embedding on the torus. Then this embedding has $|F|  = 14$ with the following possibilities:
	\begin{enumerate}
		\item $f_3 = 4$ and $f_4 = 10$, 
		\item $f_3 = 5$, $f_4 = 8$, and $f_5 =1$, 
		\item $f_3 =6$, $f_4 = 6$, and $f_5 =2$, or
		\item $f_3 = 6$, $f_4 = 7$, and $f_6 = 1$.
	\end{enumerate}
\end{proposition}

\begin{proof}
	As noted above, any toroidal embedding of $G$ must have $|F| =14$.  By Fact (4), we have that 
	
	$$52 = \sum_{i=3}^{\infty} i f_i = 3f_{3} + \sum_{i=4}^{\infty} i f_i  \geq 3f_{3} + 4\sum_{i=4}^{\infty} f_i.$$ 
	
	Thus,
	
	$$ \frac{52-3f_3}{4} \geq \sum_{i=4}^{\infty} f_i = |F| - f_{3}.$$
	
	This implies $13 + \frac{f_3}{4} \geq |F|$. Since $|F| =14$, this implies that $f_{3} \geq 4$.  We will frequently make use of this type of combinatorial restriction in this proof, though we will suppress most of the computations moving forward. Combined with Lemma \ref{lem:trianglelimits}, we have $ 4 \leq f_3 \leq 6$. 
	
	First, suppose $f_3 = 6$. Then $\displaystyle\sum_{i=4}^{\infty} f_i = 8$ and $\displaystyle\sum_{i=4}^{\infty} i f_i = 34$. As a result, at least six quadrilaterals must be used. If $f_4 = 6$, then the remaining two faces must be pentagons. If $f_4 = 7$, then the remaining face must be a hexagon. If $f_4   = 8$  then we would have $\displaystyle\sum_{i=4}^{\infty} i f_i = 32$, which is a contradiction. 
	
	Now suppose $f_3 =5$. Then $\displaystyle\sum_{i=4}^{\infty} f_i = 9$ and  $\displaystyle\sum_{i=4}^{\infty} i f_i = 37$. This forces  eight quadrilaterals to be used. Thus, $f_4 = 8$ and $f_5 =1$. 
	
	Finally, suppose $f_3 = 4$. Then $\displaystyle\sum_{i=4}^{\infty} f_i = 10$ and $\displaystyle\sum_{i=4}^{\infty} i f_i = 40$, which can only occur if all ten faces are quadrilaterals.  
\end{proof}

In light of Proposition \ref{prop:toroidalembeddingcases}, it will help to more carefully describe how faces can behave in a $2$-cell embedding of $G$ on the torus. We first introduce some necessary notation. We say that a $K_4$ fiber of $G$ is an \textbf{outer} $K_4$ if its vertices are adjacent to exactly four of the edges contained in the $P_3$ fibers. The unique $K_{4}$ fiber of $G$ whose vertices are adjacent to all edges in the $P_3$ fibers is called the \textbf{inner} $K_4$ fiber. Label the $K_4$ fibers $K^{1}$, $K^{2}$, and $K^{3}$, respectively, with $K^{2}$ inner.  We use $E_{1}$ to denote the set of edges in the $P_{3}$ fibers incident to both $K^{1}$ and $K^{2}$ and $E_{2}$ to denote the set of edges in the $P_{3}$ fibers incident to both $K^{2}$ and $K^{3}$. Given a $2$-cell embedding of $G$, we say that a quadrilateral is a \textbf{mixed quadrilateral} if its boundary is a $4$-cycle of $G$ that consists of two edges from  $E_1$ ($E_2$),  one edge from a $K^1$ ($K^3$), and the corresponding edge from $K^2$. Note, that any quadrilateral in an embedding of $G$ that contains a $P_3$ fiber edge must be a mixed quadrilateral. Up to symmetry, there are two distinct types of pentagons that bound $5$-cycles in $G$ that could occur in this $2$-cell embedding of $G$, which we denote $p_1$ and $p_2$. The 5-cycle boundary of $p_1$ contains one edge on $K^2$, two edges in either $E_1$ or $E_2$, and two edges in an outer $K_4$. Similarly the 5-cycle boundary of $p_2$ contains two edges on $K^2$, two edges in either $E_1$ or $E_2$, and one edge in an outer $K_4$. Hexagons that bound $6$-cycles in $G$ could possibly be placed in four distinct ways in a $2$-cell embedding of $G$, denoted $h_1$, $h_2$,  $h_3$, and $h_4$. The 6-cycle boundary of $h_1$ contains one edge in $K^1$, two edges in $E_1$, two edges in $E_2$, and one edge in $K^3$. The 6-cycle boundary of $h_2$ contains two edges in $K^2$, two edges in either $E_1$ or $E_2$, and two edges in an outer $K_4$. The 6-cycle boundary of $h_3$ contains one edge in $K^2$, two edges in either $E_1$ or $E_2$, and three edges in an outer $K_4$. The $6$-cycle boundary of $h_4$ contains one edge in an outer $K_4$, two edges in either $E_1$ or $E_2$, and three edges in  $K^2$.

For the above description, we only considered faces in a $2$-cell embedding of $G$ on the torus whose boundary was an appropriate cycle. We now show those are the only types of faces that can occur in such an embedding. 

\begin{lemma}\label{lem:K4P3faces}
	If $G$ admits a $2$-cell embedding on the torus, then any $n$-gon face of this embedding must be bound by an $n$-cycle of $G$.
\end{lemma}
\begin{proof}
	Suppose we are given a $2$-cell embedding of $G$ on the torus.  Let $F$ be a face in this embedding. Then the boundary of $F$ is a closed walk in $G$ determined by a ``facial walk'' of this embedding. From Proposition \ref{prop:toroidalembeddingcases}, we only need to consider when $F$ is a triangle, quadrilateral, pentagon, or hexagon. We will show that if the facial walk for $F$ is not an appropriate cycle, then we always reach a contradiction. 
	
	First, suppose our embedding admits an $n$-gon face, $F$, that is not bound by an $n$-cycle for some $n=3$, $4$, or $5$. The boundary of $F$ must be a closed walk in $G$, and since the girth of $G$ is three, this closed walk must have a subwalk of the form $eve$ for some $e \in E(G)$ and $v \in V(G)$. However, this implies that $G$ admits a degree one vertex, which is a contradiction.

	
	Now, suppose $F$ is a hexagon and the facial walk of $F$ determines a closed walk of length six in $G$ that is not a $6$-cycle.  If the facial walk for $F$ does not contain any cycle subgraphs, then we once again get a degree one vertex in $G$, which is a contradiction. However, in any $K_4$ fiber of $G$, there exist pairs of $3$-cycles that intersect along a single edge. This is the only scenario where the facial walk of $F$ could contain cycle subgraphs and not be a $6$-cycle. However, such an $F$ is not possible. If $f_6 \geq 1$, then Proposition \ref{prop:toroidalembeddingcases} shows that $f_3 = 6$ and $f_4 =7$. Note that, any triangle in this embedding can only bound edges in $K_4$ fibers, and we just assumed our one hexagon only bounds edges in $K_4$ fibers. There are only seven quadrilaterals left in this embedding and they must be mixed quadrilaterals, each of which meets exactly two $P_3$ fiber edges.   However, there are eight edges in our $P_3$ fibers, each of which must lie on the boundary of two faces in this embedding, and so, some edge in our $P_3$ fiber will only be border one face, which is a contradiction.
\end{proof}







\begin{theorem}
	\label{thm:K4P3genus}  $\gamma(K_{4} \boxempty P_{3}) =2$.
\end{theorem}

\begin{proof}
	We will first show that $G = K_{4} \boxempty P_{3}$ does not embed on the torus by analyzing the four possibilities given in Proposition \ref{prop:toroidalembeddingcases}. We then show that $G$ does admit an embedding on $S_{2}$.

	Suppose $G$ admits a $2$-cell embedding on the torus with $|F|=14$. In each case, we show that our assumptions violate Fact (2), giving a necessary contradiction. The following is an important fact we will use throughout the proof: any $e \in E_1 \cup E_2$ does not lie on a $3$-cycle and if such an edge $e$ is bound by a quadrilateral in this embedding, then it must be a mixed quadrilateral. 
	
	Suppose  $f_3 = 4$ and $f_4  = 10$.  To make sure each edge in $E_1 \cup E_2$ is incident to two faces, this embedding must have eight mixed quadrilaterals, four of which are incident to edges in $E_1$ and four of which are incident to edges in $E_2$. Each of these eight mixed quadrilaterals is incident to exactly one edge in $K^2$. Since $12 = 2 |E(K^2)|$ and each edge must border exactly two faces,  four of the (not necessarily distinct) edges in $K^2$ must be bound by faces in this embedding that are not mixed quadrilaterals. Since each edge in $E_1 \cup E_2$ already bounds two faces, there must exist a $4$-cycle in $K^2$ that is the boundary of a quadrilateral in this embedding. At this point, all faces incident to edges in $K^2$ have been determined. Thus, by Lemma \ref{lem:trianglelimits} and our assumptions on $f_3$ and $f_4$, there are two triangles incident to only edges in $K^1$, two triangles incident to only edges in $K^3$, and one quadrilateral that must be incident to either only edges in $K^1$ or $K^3$. Note that $K^1$ and $K^3$ each have four edges incident to mixed quadrilaterals. Thus, either some edges in $K^1$ do not border two faces and some edges in $K^2$ border at least three faces, or vice versa, giving a contradiction.

	Suppose $f_3 = 5$, $f_4 = 8$, and $f_5 =1$. The two types of pentagons ($p_1$ and $p_2$ described earlier in this section) each bound two edges coming from either $E_1$ or $E_2$. As a result, we must then have seven mixed quadrilaterals in this embedding to make sure each edge in $E_1 \cup E_2$ borders exactly two faces. This leaves one quadrilateral and five triangles to complete the embedding. Since each mixed quadrilateral is incident to an edge in $K^2$ and a pentagon is incident to either one or two edges in $K^2$, either four or three more (not necessarily distinct) edges in $K^2$ must be incident to the other faces in this embedding. By Lemma \ref{lem:trianglelimits} and $f_3 =5$, either one or two triangles are incident to only edges in $K^2$. Having two triangles bordered by edges in $K^2$ results in some edge in $K^2$ having at least three faces incident to it, which is a contradiction. So, suppose there is only one triangle bordered by edges in $K^2$, and so, each outer $K_4$ is bordered by two triangles. In this case, a $p_2$ pentagon must occur to make sure each edge of $K^2$ is bordered by exactly two faces. This also means the one remaining quadrilateral must be incident to only edges in a single outer $K_4$, say $K^1$. Then some edge in $K^1$ has three faces incident to it. Thus, an embedding of $G$ with $f_3 = 5$, $f_4 = 8$, and $f_5 =1$ is not possible.

	Suppose  $f_3 = 6$, $f_4 = 6$, and $f_5 = 2$. By Lemma \ref{lem:trianglelimits}, each $K_4$ fiber contains exactly two triangles. Note that, each pentagon borders a pair of edges in either $E_1$ or $E_2$. Thus, all six quadrilaterals must be mixed quadrilaterals in order to make sure each edge in $E_1 \cup E_2$ borders exactly two faces. Each mixed quadrilateral borders exactly one edge in $K^2$, each pentagon borders either one or two edges in $K^2$, and there are two triangles bordering a total of six edges in $K^2$. As a result, a minimum of $14$ edges in $K^2$ are bordered by faces, which forces some edge in $K^2$ to border at least three faces, providing a contradiction. 
	
	Suppose  $f_3 = 6$, $f_4 = 7$, and $f_6 =1$. By Lemma \ref{lem:trianglelimits}, each $K_4$ fiber contains exactly two triangles. As discussed before the proof, there are four types of hexagons that can exist in an embedding of $G$. We now break this case into  sub-cases, depending on what type of hexagon is used. 
	
	First, suppose an $h_1$ hexagon is used in this embedding. Since this hexagon borders four edges from $E_1 \cup E_2$, we must have six mixed quadrilaterals in this embedding to ensure every edge in $E_1 \cup E_2$ borders exactly two faces. Thus, we only need to determine where the last quadrilateral is placed in this embedding. This quadrilateral must be bounded by edges in a $K_4$ fiber since the union of the pentagon and six mixed quadrilaterals collectively border every edge in $E_1 \cup E_2$ twice. However, no matter which $K_4$ fiber is used here, we will have an edge incident to at least three faces, providing a contradiction.
	
	Now, suppose an $h_2$, $h_3$, or $h_4$ hexagon is used in this embedding. Such a hexagon does not border one of the outer $K_4$ fibers, say $K^3$. In addition, such a hexagon borders exactly two edges from either $E_1$ or $E_2$. As a result, all seven quadrilaterals must all be mixed quadrilaterals to ensure every edge in $E_1 \cup E_2$ borders exactly two faces. Four of these mixed quadrilaterals will each be incident to one edge in $K^3$ and the other three mixed quadrilaterals do not border $K^3$. At this point, we have determined where all our faces are placed and we see that some edges in $K^3$ are bordered by at most one face, which is a contradiction.

	We now describe how to embed $G$ on $S_2$. First, embed each $K_4$ fiber on a separate $2$-sphere as the $1$-skeleton of a tetrahedron. We use $F_i$ for $i=1,\ldots, 4$ to denote the four triangles in one of these embeddings. First, delete the interiors of $F_1$ and $F_2$ from the $K^{1}$ embedding and from the $K^{2}$ embedding. Insert a tube (homeomorphic to $\mathbb{S}^{1} \times [0,1]$),  between the $F_j$ face deleted from $K^{1}$ and the $F_j$ face deleted from $K^{2}$ for $j=1,2$. The four edges of $E_{1}$ can be embedded on these two tubes. Now, delete the triangular faces $F_3$ and $F_4$ coming from the $K^{2}$ embedding on a $2$-sphere and also delete $F_3$ and $F_4$ from the $K^{3}$ embedding on a $2$-sphere. Insert two tubes in a similar fashion to embed the edges of $E_{2}$, completing the embedding of $G$. Since this embedding involves three disjoint $2$-spheres joined by four tubes, the resulting genus of the surface is two. 
\end{proof}


\section{Restricting toroidal $G \boxempty H$ via vertex connectivity.}
\label{sec:connectedrestricted}

The \textbf{vertex-connectivity} of a graph $G$, denoted $\kappa(G)$, is the minimum cardinality of a vertex-cut if $G$ is not complete or $\kappa(G) =n-1$ if $G=K_n$ for $n\in\mathbb{N}$. A connected graph $G$ is \textbf{$k$-connected} for $1 \leq k \leq \kappa(G)$. Our goal in this section is to examine how vertex-connectivity can affect whether or not $G \boxempty H$ is toroidal. We first note that $G$ and $H$  can  be at most $4$-connected, if $G \boxempty H$ is toroidal. 

\begin{proposition}
	\label{prop:5conn}
	If $\gamma(G \boxempty H) \leq 1$, then $\kappa(G), \kappa(H) \leq 4$. 
\end{proposition}

\begin{proof}
	We proceed by contrapositive. Suppose $\kappa(G) \geq 5$. Since $H$ is a simple connected graph with $|V(H)| \geq 2$, we see that $H$ contains a $P_2$ subgraph. Note that, $|V(G \boxempty P_2)| = 2|V(G)|$ and $|E(G \boxempty P_2)| = 2|E(G)| + |V(G)|$. Since $G$ is simple, the girth of $G \boxempty P_2$ is at least three. By Proposition \ref{prop:Eulerbound}, we have that 
	
	$$\gamma(G \boxempty P_2) \geq 1 + \frac{2|E(G)| + |V(G)|}{6} - \frac{2|V(G)|}{2} = 1 + \frac{1}{3}|E(G)| - \frac{5}{6}|V(G)|.$$
	
	Since $G$ is $5$-connected, for all $v \in V(G)$, we have $deg(v) \geq 5$. By the Handshaking Lemma, $\sum deg(v) = 2|E(G)|$. Thus, $\frac{5}{2}|V(G)| \leq |E(G)|$, and so, we can conclude that 
	
	$$ \gamma(G \boxempty P_2) \geq 1 + \frac{1}{3}\left(\frac{5}{2}|V(G)|\right) - \frac{5}{6}|V(G)| = 1.$$
	
	Equality holds if and only if $G$ is $5$-regular and $G \boxempty P_2$ admits a  triangular emdedding. However, any edge in a $P_2$ fiber is not contained in any $3$-cycle of $G \boxempty P_2$, and so, a triangular embedding of $G \boxempty P_2$ does not exist. Corollary \ref{cor:minorlowerbound} implies that $\gamma(G \boxempty H) \geq \gamma(G \boxempty P_2) > 1.$
\end{proof}

We now work to show that severe restrictions can be placed on $H$ when $\kappa(G) \in \{3,4\}$. We first need a basic fact about $3$-connected graphs.

\begin{lemma}
	\label{lem:Tutte3connected}
	Let $G$ be a $3$-connected graph. Then $G$ contains a $K_4$ minor.
\end{lemma}

\begin{proof}
	The work of Tutte \cite{Tu1961} shows that any simple $3$-connected graph contains a wheel graph $W_n$ as a minor, for some $n \geq 3$. Whenever $m \leq n$, it is easy to see that $W_m$ is a minor of $W_n$:  one can first delete an appropriate number of the spoke edges of $W_n$ that connect to the central vertex, and then, contract some of the edges along the $n$-cycle of $W_n$ to obtain a $W_m$ minor. Thus, $W_3 = K_4$ is a minor of any simple $3$-connected graph, as needed. 
\end{proof}

	This result has immediate consequences for classifying toroidal Cartesian products where one factor is $3$-connected.
	
	\begin{proposition}
		\label{thm:Hrestricted}
		Let $G$ be a $3$-connected graph. If $\gamma(G \boxempty H) \leq 1$, then $H = P_2$.
	\end{proposition}
	
	\begin{proof}
		We proceed by contrapositive. Suppose $H$ is a (simple, connected, non-trivial) graph where $H \neq P_2$. Then $H$ contains a $P_3$ subgraph. By Lemma \ref{lem:Tutte3connected}, we know that $G$ has a $K_4$ minor. Then by Theorem \ref{thm:K4P3genus} and Corollary \ref{cor:minorlowerbound}, we have that $\gamma(G \boxempty H) \geq \gamma(K_{4} \boxempty P_3) = 2$. 
	\end{proof}
	
	Combined, Proposition \ref{thm:Hrestricted} and Proposition \ref{prop:5conn} give rise to the following questions: Does there exist an infinite set of $3$-connected ($4$-connected) graphs $\{G_i\}$ such that each $G_i \boxempty P_2$ is toroidal? If so, can we classify such $G_i$?


	\section{Classifying toroidal $G \boxempty P_2$, where $\kappa(G) \in \{3,4\}$.}
	\label{sec:constructing3conP2}
	
	In this section, we fully address both of questions from above. Recall that a graph $G$ is \textbf{outer-cylindrical}  if it embeds in the plane or the $2$-sphere such that there are two distinct faces whose boundaries collectively contain all the vertices of $G$. Since the class of outer-cylindrical graphs is closed under taking minors, the Graph Minor Theorem of Robertson and Seymour \cite{RoSe2004} implies that there is a finite list of minor-minimal obstructions for this class of graphs, which were then determined by  Archdeacon--Bonnington--Dean--Hartsfield--Scott \cite{ArBo2001}. The following proposition motivates our interest in outer-cylindrical graphs.
	
	\begin{proposition} \label{prop:OCboxP2}
		If $G$ is an outer-cylindrical graph, then $\gamma(G \boxempty P_2) \leq 1$.
	\end{proposition}
	
	\begin{proof}
		We first construct disjoint outer-cylindrical embeddings of the two $G$ fibers (call them $G_1$ and $G_2$) of $G \boxempty P_2$ in the $2$-sphere. Take an outer-cylindrical embedding of $G$ as $G_1$, and reflect $G_1$ over a line that does not intersect $G_1$ to build the  outer-cylindrical embedding of $G_2$. All the vertices of  each $G_i$ for $i=1,2$ are contained in two faces of this embedding, and one of these faces, which we call the outer face, is common between $G_1$ and $G_2$. For each vertex $v$ in $G_1$ on this outer face, we can embed on edge on the $2$-sphere connecting $v$ to the corresponding vertex in $G_2$. We can then cut out the interiors of the ``inner'' faces for $G_1$ and $G_2$ and add a tube between them carrying the remaining edges of $G \boxempty P_2$ to build an embedding of $G \boxempty P_2$ on the torus. See Figure \ref{figOCP2embedding} for a schematic of this embedding. \end{proof}
	
	\begin{figure}[h!]
		\centering
		\begin{overpic}[scale=0.7]{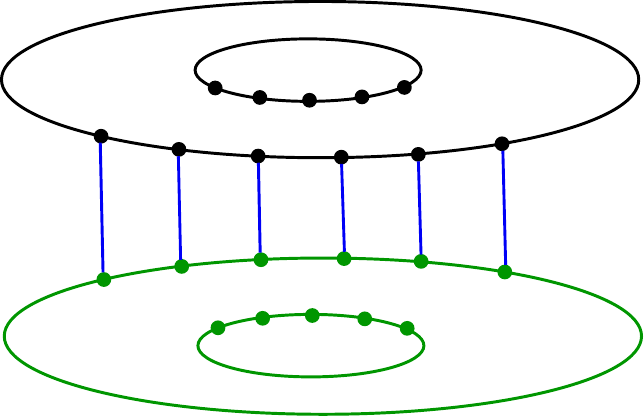}
			\put(15, 10){$G_{1}$}
			\put(15,53){$G_{2}$}
		\end{overpic}
		\caption{The inner and outer faces for the $G_{1}$ fiber are depicted in green and the inner and outer faces for the $G_{2}$ fiber are depicted in black. The edges contained in the $P_{2}$ fibers that connect to vertices on the two outer faces are depicted in blue. The edges contained in the $P_{2}$ fibers that connect to vertices on the inner faces are not depicted.}
		\label{figOCP2embedding}
	\end{figure}    

In fact, a converse to Proposition \ref{prop:OCboxP2} also holds. 

\begin{proposition}
    \label{prop:toroidalimpliesOC}
    Suppose $G$ is a (connected) graph. If $\gamma(G \boxempty P_2) \leq 1$, then $G$ is an outer-cylindrical graph. 
\end{proposition}

\begin{proof}
Let $G$ be a (connected) graph such that $\gamma(G \boxempty P_2) \leq 1$. First, if $\gamma(G \boxempty P_2) = 0$, then the classification of planar Cartesian products of graphs given by Behzad and Mahmoodian \cite{BeMa1969} implies that $G$ is outerplanar, and so, $G$ is outer-cylindrical. Moving forward, we assume $\gamma(G \boxempty P_2) = 1$, which implies that $G$ is not outerplanar. Furthermore, there exists a minimal genus embedding of $G \boxempty P_2$ on the torus, $\mathbb{T}$, which must be a $2$-cell embedding. We work with this fixed embedding for the remainder of this proof. Let $G^{+}$ and $G^{-}$ be the two $G$-fiber subgraphs of $G \boxempty P_2$. Consider the induced embedding of $G^{+}$ on $\mathbb{T}$ and the collection of regions $\mathbb{T} \setminus G^{+} = \{R_{1}, \ldots, R_{k}\}$. There are two possible cases to consider depending on the topology of these regions. 

\underline{Case 1:} Suppose each region $R_{i}$, for $1 \leq i \leq k$, is a $2$-cell region.  Since $G^{-}$ is connected, it must be embedded in the interior of some $2$-cell region $R_{i}$.  Since $G^{+}$ is not outerplanar, there exists some vertex $v^{+} \in V(G^{+})$ that is not on the boundary walk $W$ of the region $R_{i}$. However, there exists an edge $e = v^{+}v^{-}$ in a $P_2$ fiber connecting $v^{+}$ to the corresponding vertex $v^{-} \in V(G^{-})$, and clearly, $W$ provides an obstruction for $e$ to exist in such an embedding. This gives a contradiction. 

\underline{Case 2:} Suppose there exists some region, say $R_{j}$, that is a non-$2$-cell region. First, $R_{j}$ must be the unique non-$2$-cell region in this collection, since otherwise, $G$ would not be connected. Then either (i) $R_j$ is homeomorphic to a once-punctured torus or (ii) $R_j$ is homeomorphic to an open annulus. Sub-case (i) implies that $G^{-}$ must be embedded in $R_j$, since otherwise, $\gamma(G \boxempty P_2) = 0$. This also implies that $G^{+}$ is embedded on a closed disk in $\mathbb{T}$ whose boundary is some walk in $G^{+}$. However, we can now apply the same argument used in Case 1 to see that there is a vertex on the interior of this closed disk that must connect via an edge to a vertex in the complement of this closed disk, but such an edge would not exist in this embedding, giving a contradiction for this sub-case. So, suppose $R_j$ is homeomorphic to an annulus. Then $G^{-}$ must be embedded in $R_j$, since otherwise, $\gamma(G \boxempty P_2) = 0$. Let $W_1$ and $W_2$ be the two boundary closed walks of $R_j$. Similar to the argument in Case 1, if there exists  $v^{+} \in V(G^{+})$ that is not on $W_1 \cup W_2$, then $W_1 \cup W_2$ forbids the existence of the edge $e = v^{+}v^{-}$, giving a contradiction. Thus every vertex of $G^{+}$ must lie on $W_1 \cup W_2$, and so, $G$ is outer-cylindrical, as needed. 
\end{proof}

Collecting our major results provides a proof of our classification theorem that was stated in the introduction.


\begin{proof}[Proof of Theorem~\ref{thm:classification}]
Let $G$ be a $3$-connected graph. For the forwards direction,  Propositions \ref{thm:Hrestricted} and  \ref{prop:toroidalimpliesOC} collectively give the desired conclusions. For the backwards direction, Proposition \ref{prop:OCboxP2} gives the desired conclusion. 
\end{proof}

    To conclude this section, we provide one set of examples of Cartesian products of graphs that are toroidal.

	\begin{example}\label{Example1}
	    For $n \geq 3$, define $C^{2}_{2n}$ as the graph obtained from the $2n$-cycle $C_{2n}$ by adding an edge between every pair of vertices that are a distance two apart in $C_{2n}$. See Figure \ref{figC2_6} for a standard planar embedding of $C^{2}_{6}$, which also shows this graph is outer-cylindrical. For $n \geq 3$, $C^{2}_{2n}$ is known to be $4$-connected; see \cite{Ma1980}. Thus, $\{C^{2}_{2n} \boxempty P_{2}\}_{n=3}^{\infty}$ provides an infinite class of toroidal Cartesian products of graphs where one factor is $4$-connected. Furthermore, by Propositions \ref{prop:5conn} and \ref{thm:Hrestricted}, we see that $G = C^{2}_{2n}$ realizes the maximal vertex connectivity for toroidal $G \boxempty P_2$.  
	
	\begin{figure}[h!]
		\centering
		\begin{overpic}[scale=0.3]{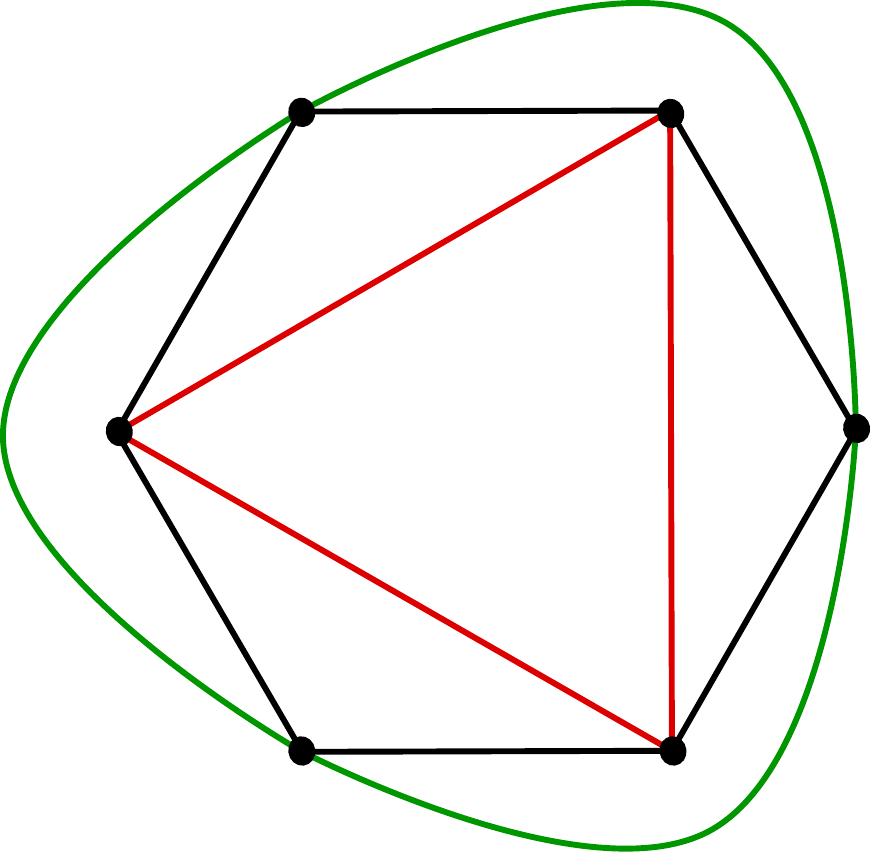}
			\put(30,89){$v_{1}$}
			\put(77,89){$v_{2}$}
			\put(101,49){$v_{3}$}
			\put(79,10){$v_{4}$}
			\put(27, 10){$v_{5}$}
			\put(7,48){$v_{6}$}
		\end{overpic}
		\caption{The standard planar embedding of $C^{2}_{6}$ so that the green cycle $v_{1}v_{3}v_{5}$ and red cycle $v_{2}v_{4}v_{6}$ each bound faces for an outer-cylindrical embedding.}
		\label{figC2_6}
	\end{figure}
	\end{example}
	


\begin{acknowledgements}
We would like to thank Doug Rall and Thomas Mattman for feedback and suggestions related to this article. 
 \end{acknowledgements}

\section*{Funding}
This work was financially supported by the Furman University Department of Mathematics  via the Summer Mathematics Undergraduate Research Fellowships.

\section*{Availability of data and materials}
Not applicable. 

\section*{Conflict of Interests}
The authors have no relevant financial or non-financial interests to disclose.

\bibliographystyle{hamsplain}
\bibliography{biblio}

\end{document}